\documentclass[a4paper]{amsart}

\usepackage{amsmath}
\usepackage{amssymb}

\newtheorem{MainThm}{Theorem}

\newtheorem{thm}{Theorem}[section]

\newtheorem{prop}[thm]{Proposition}
\theoremstyle{definition}

\theoremstyle{remark}
\newtheorem{rem}[thm]{Remark}
\numberwithin{equation}{section}

\newcommand{\bP}{\mathbb{P}}

\newcommand{\bQ}{\mathbb{Q}}
\newcommand{\bR}{\mathbb{R}}
\newcommand{\bS}{\mathbb{S}}

\newcommand{\bZ}{\mathbb{Z}}
\newcommand{\gM}{\bold{M}}
\newcommand{\gO}{\bold{O}}
\newcommand{\gS}{\bold{S}}
\newcommand{\gT}{\bold{T}}
\newcommand{\cL}{\mathcal{L}}
\newcommand{\cX}{\mathcal{X}}

\newcommand{\MTO}{\gM \gT \gO }
\newcommand{\MTSO}{\gM \gT \gS \gO}

\newcommand\trf{\mathrm{trf}}
\newcommand\Diff{\mathrm{Diff}}
\newcommand\nmcg[1]{\mathcal{N}_{#1}}

\newcommand{\id}{\operatorname{id}}

\newcommand{\suspinf}{\Sigma^{\infty}}
\newcommand{\loopinf}{\Omega^{\infty}}

\title[Divisibility of characteristic classes]{On the divisibility of characteristic classes of non-oriented surface bundles}

\author{Johannes Ebert, Oscar Randal-Williams}

\address{
Mathematical Institute\\
24-29 St Giles'\\
Oxford\\
OX1 3LB\\
England
}
\email{ebert@maths.ox.ac.uk\\
   randal-w@maths.ox.ac.uk}


\thanks{J. Ebert is supported by a fellowship within the Postdoc-Programme of the German 
Academic Exchange Service (DAAD); \\O. Randal-Williams is supported by an EPSRC Studentship, DTA grant number
EP/P502667/1}

\begin{document}

\begin{abstract}
In this note we introduce a construction which assigns to an arbitrary manifold bundle its fiberwise orientation covering. This is used to show that the zeta classes of unoriented surface bundles are not divisible in the stable range.
\end{abstract}

\maketitle

\section{Introduction}
The mapping class group $\nmcg{g}$ of a non-orientable surface $S_g$ of genus $g$ (that is, the connected sum of $g$ copies of $\mathbb{RP}^2$) is defined to be
$$\nmcg{g} := \pi_0(\Diff(S_g))$$
the group of components of the diffeomorphism group of that surface. If $g \geq 3$, the components of $\Diff(S_g)$ are contractible \cite{EE}, hence $B \nmcg{g} \simeq B \Diff(S_g)$, and so the cohomology of $B \nmcg{g}$ (or the group cohomology of $\nmcg{g}$) can be interpreted as the ring of characteristic classes for $S_g$-bundles.

Wahl \cite{Wahl} has proved a homological stability theorem for these groups, which says that in degrees $* \leq (g-3)/4$ the cohomology groups $H^*(\nmcg{g})$ are independent of the genus $g$. We call this range of degrees the \textit{stable range}. With rational coefficients these stable groups can be identified: there are certain integral characteristic classes $\zeta_i$ in degree $4i$ and the map
$$\bQ [ \zeta_1, \zeta_2, \zeta_3,...] \to H^*(\nmcg{g} ; \bQ)$$
is an isomorphism in the stable range. In \cite{R-W} the second author calculates these stable groups with coefficients in a finite field, and tabulates some low-dimensional integral groups.

The classes $\zeta_i$ are analogues of the even Miller--Morita--Mumford classes, for unoriented surface bundles. Galatius, Madsen and Tillmann \cite{GMT} have studied the divisibility of the Miller--Morita--Mumford classes $\kappa_i \in H^*(\Gamma_\infty ; \bZ)$ in the free part of the integral cohomology of the stable mapping class group $\Gamma_\infty$. They find that the even classes are divisible by 2 and the odd classes are divisible by a denominator of a Bernoulli number. In \cite{Eb2} the first author studied the divisibility of the Miller--Morita--Mumford classes for surface bundles with spin structures, and it was shown that the divisibility increases by a certain power of 2 relative to the non-spin case. Continuing the study of divisibility of characteristic classes of surface bundles, we prove
\begin{MainThm}\label{Main}
The universal zeta classes, $\zeta_n \in H^{4n}(\nmcg{g} ; \bZ)$, are not divisible in the stable range. Indeed, they are not divisible in the free quotient of cohomology $H^{4n}_{free}(\nmcg{g} ; \bZ)$ in this range.
\end{MainThm}
This gives the trend that extra structure on the vertical tangent bundle, such as an orientation or a spin structure, gives extra divisibility of characteristic classes of surface bundles.

\section{Lifting diffeomorphisms to orientation coverings}

In this section, we will construct a natural homomorphism from the diffeomorphism group $\Diff(M)$ of a smooth $d$-manifold to the group $\Diff^+ (\tilde{M})$ of orientation-preserving diffeomorphisms of the orientation covering of $M$. This implies that any smooth fiber bundle $p:E \to B$ admits a two-fold covering $\pi:\tilde{E} \to E$, such that $p \circ \pi:\tilde{E} \to B$ is an oriented smooth fiber bundle and that the restriction of $\pi:\tilde{E} \to E$ to a fiber of $p$ is the orientation covering.

Let $M$ be a smooth $d$-manifold, $d > 0$, and let $\Lambda^d TM$ be the highest exterior power of the tangent bundle, which is a real line bundle. The total space of the orientation covering of $M$ can be defined as 
\begin{equation}\label{deforcov}
\tilde{M}:=(\Lambda^d TM \setminus 0)/ \bR_{>0}.
\end{equation}
The canonical map $\pi:\tilde{M}\to M$ is a two-sheeted covering. The space $\tilde{M}$ is a smooth manifold, which is orientable. In fact, there is a preferred orientation of $\tilde{M}$. To see this, recall that an orientation of a $d$-dimensional real vector space $V$ is a component of $\Lambda^d V  \setminus 0$, or, in other words, one of the two points of $(\Lambda^d V \setminus ) / \bR_{>0}$. Thus a point in $x \in \tilde{M}$ is by definition an orientation of the tangent space $T_{\pi(x)} M$. On the other hand, the map $\pi$ is a covering and hence a local diffeomorphism. The differential $T_x \pi$ at $x \in \tilde{M}$ is a linear isomorphism $T_x \tilde{M} \to T_{\pi(x)M}$. It follows that $T_{x} \tilde M$ has a preferred orientation: the point $x$ defines an orientation of $T_{\pi(x)} M$ and therefore one of $T_x \tilde{M}$ via the linear isomorphism $T_x \pi$. Using local coordinates on $M$, it is easy to see that the orientations of the tangent spaces $T_x \tilde{M}$ constructed above fit continuously together and define an orientation of $\tilde{M}$, the \emph{preferred orientation} of $M$.

Moreover, this construction is natural: a diffeomorphism $f:M \to N$ of smooth manifolds induces a diffeomorphism $\tilde{f}:\tilde{M} \to \tilde{N}$ which covers $f$. It is easy to see that $\tilde{f}$ is orientation-preserving provided $\tilde{M}$ and $\tilde{N}$ are endowed with the preferred orientations.
If $g:N \to P$ is another diffeomorphism, then $\widetilde{g \circ f} = \tilde{g} \circ \tilde{f}$. Also, $\tilde{\id_M}= \id_{\tilde{M}}$. Finally, we did not use that $f$ is a diffeomorphism, but only that the differential of $f$ was nonsingular. It follows that the assignments $M \mapsto \tilde{M}$ and $f \mapsto \tilde{f}$ define a functor $\cL$ from the category $\cX_d$ of smooth $d$-manifolds and local diffeomorphisms to the category $\cX_{d}^{+}$ of oriented $d$-manifolds and orientation-preserving local diffeomorphisms. In particular, we defined a group homomorphism $\cL_M: \Diff(M) \to \Diff^{+}(\tilde{M})$.

For a manifold $M$, we denote by $\pi_{M}$ the covering map $\tilde{M} \to M$ and by $\iota_M:\tilde{M} \to \tilde{M}$ the unique nontrivial deck transformation. If $f:M \to N$ is a (local) diffeomorphism, the following relations hold
\begin{equation}\label{relations} 
\pi_N \circ \tilde{f}= f \circ \pi_M; \tilde{f} \circ \iota_M = \iota_N \circ \tilde{f}.
\end{equation}
The morphism spaces of the categories $\cX_d$ and $\cX_{d}^{+}$ have a natural topology, the weak $C^{\infty}$-topology, with respect to which the composition maps are continuous. Thus $\cX_d$ and $\cX_{d}^{+}$ are topological categories. Using local coordinates, it is easy to see that the functor $\cL$ is continuous. In particular, the homorphism $\cL_M: \Diff(M) \to \Diff^{+}(\tilde{M})$ is continuous.

Let us now discuss smooth fiber bundles. Let $p:E \to B$ be a smooth fiber bundle with fiber a $d$-dimensional smooth manifold $M$ and structural group $\Diff(M)$ (with the weak $C^{\infty}$-topology). Consider the associated $\Diff(M)$-principal bundle $Q \to B$, which has the property that $Q \times_{\Diff(M)} M \cong E$. Via the homomorphism $\cL_M$, the manifold $\tilde{M}$ has a $\Diff(M)$-action by orientation preserving diffeomorphisms. Hence the fiber bundle
\[
q:\tilde{E}:= Q \times_{\Diff(M)} \tilde{M} \to M
\]
is an oriented smooth fiber bundle with fiber $\tilde{M}$. Because of (\ref{relations}), there is a twofold covering $\pi_E : \tilde{E} \to E$, such that $q = p \circ \pi_E$. Furthermore, there is a fiber-preserving and orientation-reversing involution $\iota_E$ on $\tilde{E}$. We call $\tilde{E}$ the \emph{fiberwise orientation cover} of $E$. We summarize the results of this section.

\begin{thm}\label{fibwiseorcov}
The fiberwise orientation cover $\pi:\tilde{E} \to E$ of a smooth fiber bundle $p:E \to B$ is a two-sheeted covering whose restriction to any fiber $E_b$ of $p$ is the orientation cover of $E_b$. The composition $q = p \circ \pi_E$ is an oriented fiber bundle. Furthermore, $\tilde{E}$ and $\pi_E$ are uniquely determined by these properties (up to orientation-preserving isomorphism).
\end{thm}

We conclude with a simple remark. All the constructions in this section make sense when the manifold $M$ (or the fiber bundle $E$) is orientable. If this is the case, then $\tilde{M}$ is the disjoint sum of two copies of $M$. The choice of an orientation of $M$ singles out a component of $\tilde{M}$.

\section{Characteristic classes of surface bundles}

In this section, we give a brief review of the theory of characteristic classes of surface bundles, both oriented and non-oriented. First we discuss the oriented case. Let $\pi:E \to B$ be an oriented surface bundle and let $T_v E$ be the vertical tangent bundle. It is an oriented $2$-dimensional real vector bundle on $E$ and thus it has an Euler class $e(T_v(E)) \in H^2 (E; \bZ)$. We can consider $T_v E$ also as a complex line bundle (there is a complex structure on it, which is unique up to isomorphism) and the Euler class agrees with the first Chern class. The Miller--Morita--Mumford classes are defined to be
\[
\kappa_n (E) := \pi_{!} (e(T_vE)^{n+1}) \in H^{2n} (B; \bZ),
\] 
where $\pi_{!}:H^* (E; \bZ) \to H^{*-2}(B; \bZ)$ is the umkehr, or cohomological fiber-integration, map. This definition cannot be generalized to the non-oriented case without further effort, because both the Euler class and the umkehr map only exist for oriented surface bundles.

The concept needed for a generalization is the Becker--Gottlieb transfer \cite{BG}. Let $p:E \to B$ be a smooth fiber bundle with compact fibers diffeomorphic to $F$ (not necessarily of dimension $2$). The transfer is a stable map in the converse direction, more precisely, it is a map of the suspension spectra
\[
\trf_p : \suspinf B_+ \to \suspinf E_+.
\]
Recall that the spectrum cohomology of the suspension spectrum of a space $\suspinf X_+$ agrees with the ususal cohomology of the space $X$. Thus we can form the map $\trf_{p}^{*} \circ p^{*}: H^*(B;\bZ) \to H^*(B;\bZ)$, and for all $x \in H^*(B ; \bZ)$ we have 
\begin{equation}\label{eulernumber}
\trf_{p}^{*} \circ p^{*}(x) = \chi(F) \cdot x,
\end{equation}
where $\chi(F)$ denotes the Euler number of the fiber (\cite[Theorem 5.5]{BG}). Furthermore, if $q: \tilde{E} \to E$ is another smooth fiber bundle with compact fibers, then $p \circ q$ is also such a fiber bundle. In this situation the composition of the transfers is homotopic to the transfer of the composition (see \cite[Equation 2.2, page 137]{BM}):
\begin{equation}\label{transhomotop}
\trf_{p \circ q} \simeq \trf_q \circ \trf_p.
\end{equation}
A diffeomorphism $f:M \to N$ of manifolds can be considered as a fiber bundle whose fiber is a point. By (\ref{eulernumber}),
\begin{equation}\label{transfdiff}
\trf_{f}^* \circ f^* = \id_{H^*(N;\bZ)}, f^* \circ \trf^{*}_{f} = \id_{H^*(M;\bZ)}.
\end{equation}
In fact, $\trf_f$ and $\suspinf (f^{-1})$ are homotopic, but we do not need this fact. The transfer of an oriented fiber bundle $p:E \to B$ is closely related to the umkehr map. For all $x \in H^*(E; \bZ)$, one has (see \cite[Theorem 4.3]{BG})
\begin{equation}\label{transumkehr}
\trf_{p}^{*} (x)= p_{!} ( x \cup e(T_v E) ).
\end{equation}
The identity (\ref{transumkehr}) implies
\begin{equation}\label{mumtrans}
\kappa_n(E) = \trf^{*}_{p} (e(T_v E)^n)
\end{equation}
for the Miller--Morita--Mumford classes of an oriented surface bundle $p:E \to B$. Because of the identity $p_1 (L) = e(L)^2$ for the Pontrjagin class of a $2$-dimensional oriented real vector bundle $L$, we see that
\begin{equation}\label{kappapontr}
\kappa_{2n}(E) = \trf^{*}_{p} (p_1(T_v E)^n).
\end{equation}
This generalises to the non-oriented case. Wahl defines (\cite[page 3]{Wahl})
\begin{equation}\label{zetaclass}
\zeta_i (E) := \trf_{p}^{*} (p_1 (T_v E)^{i}) \in H^{4i} (B;\bZ),
\end{equation}
for a non-oriented surface bundle $p:E \to B$, where $p_1 (T_vE) \in H^4 (E;\bZ)$ is the first Pontrjagin class of the vertical tangent bundle. Now we can state and prove the main result of this section.
\begin{thm}\label{thm1}
Let $p:E \to B$ be a non-oriented surface bundle with compact fibers and let $c: \tilde{E} \to E$ be its fiberwise orientation covering. Denote $q := p \circ c: \tilde{E} \to B$. Then the following relations hold:
\begin{enumerate}
\item For all $n \geq 0$, we have $\kappa_{2n} (\tilde{E}) =2 \cdot \zeta_n (E)$.
\item For all $n \geq 0$, we have $2 \cdot \kappa_{2n+1} (\tilde{E}) =0$.
\end{enumerate}
\end{thm}

\begin{proof} For the identity (1), we compute
\begin{eqnarray*}
\kappa_{2n}(\tilde{E}) & = & \trf_{q}^{*} ((p_1 (T_v \tilde{E})^n) \\
& = & \trf_p^*(\trf_c^*(c^*(p_1(T_v E)^n))) \\
& = & \trf_p^*(2 \cdot p_1(T_v E)^n) \\
& = & 2 \cdot \zeta_n(E).
\end{eqnarray*}
The first eqality is (\ref{kappapontr}). Because $c:\tilde{E} \to E$ is a smooth covering in every fiber, $c^*(T_vE) \cong T_v (\tilde{E})$, whence $p_1 (T_v \tilde{E}) = c^*(p_1(T_v E))$. Together with (\ref{transhomotop}), this fact implies the second equality. Because $c$ is a double covering, the Euler number of its fiber is $2$. Thus $\trf^{*}_{c} \circ c^* = 2$, which gives the third equality. The fourth equality is the definition.

For the proof of identity (2), we use the orientation-reversing involution $\iota$ on $\tilde{E}$. By (\ref{transfdiff}), $\trf_{\iota}^{*}= (\iota^*)^{-1}= \iota^*$. Because $c \circ \iota =c$, it follows that $\trf_{c}^{*} = \trf_{c}^{*} \circ \trf_{\iota}^{*} =  \trf_{c}^{*} \circ{\iota}^{*} $. Because $\iota$ is an orientation-reversing fiberwise diffeomorphism, it induces an orientation-reversing vector bundle isomorphism $d \iota: T_v \tilde{E} \to \iota^* T_v \tilde{E}$. Thus $e(T_v \tilde{E})= - \iota^* e (T_v\tilde{E})$. Thus
\begin{eqnarray*}
\kappa_{2n+1}(\tilde{E})  & = & \trf_{p}^{*} \trf_{c}^{*} ( e(T_v(\tilde{E})^{2n+1}) \\
& = & \trf_{p}^{*} \trf_{c}^{*} \iota^* ( e(T_v(\tilde{E})^{2n+1})\\
& = & (-1)^{2n+1} \trf_{p}^{*} \trf_{c}^{*} ( e(T_v(\tilde{E})^{2n+1})\\
& = & - \kappa_{2n+1}(\tilde{E}).
\end{eqnarray*}
\end{proof}

\begin{rem}
An implication of this theorem is that for an oriented surface bundle $E' \to B$, the characteristic classes $2 \cdot \kappa_{2n+1}(E')$ are obstructions to $E'$ admitting an orientation-reversing fixed-point free fiberwise involution. Furthermore, for bundles which do admit such an involution, it gives an interpretation of $\frac{1}{2} \kappa_{2n}(E')$ as the zeta classes of the associated quotient bundle of non-orientable surfaces.
\end{rem}

\section{An example}

In this section, we consider the rather easy example of a genus zero surface bundle. Let $\gamma_3 \to BSO(3)$ be the universal $3$-dimensional oriented Riemannian real vector bundle and let $\bS(\gamma_3) \to BSO(3)$ be its unit sphere bundle. It is known that this is the universal smooth oriented bundle with fiber $\bS^2$, but we do not need this fact. It is not hard to see that $\kappa_{2n}(\bS(\gamma_3))= 2 p_{1}(\gamma_3)^n$, compare \cite[page 49]{Eb}. The bundle $\bS(\gamma_3)$ admits an orientation-reversing, fixed-point free involution on its fibers, namely the antipodal map $- \id$. The quotient is $\bP(\gamma_3)$, the $\bR \bP^2$-bundle associated to $\gamma_3$. By Theorem \ref{thm1}, we have
\begin{equation}
2 \zeta_n (\bP(\gamma_3)) = 2 p_{1}(\gamma_3)^n.
\end{equation}
The integral cohomology ring of $BSO(3)$ is not too hard to compute. Denote by $\chi \in H^3 (BSO(3); \bZ)$ the universal Euler class. Then we have
\begin{equation}
H^*(BSO(3);\bZ) \cong \bZ[p_1, \chi]/(2\chi).
\end{equation}
To see this, one uses the Leray--Serre spectral sequence of the fibration $\bS^2 \to BSO(2) \to BSO(3)$. In particular, the powers $p_{1}^{n}(\gamma_3)$ are not divisible in the free quotient $H_{free}^{*}(BSO(3);\bZ)$ of $H^*(BSO(3) ; \bZ)$. We have shown:

\begin{prop}\label{g0notdiv}
The class $ \zeta_n (\bP(\gamma_3)) $ is not divisible in $H_{free}^{*}(BSO(3);\bZ)$.
\end{prop}

\begin{rem}
It may look a bit confusing that the universal $\bR \bP^2$-bundle has base space $BSO(3)$, but this is indeed the case. The reason is the isomorphism of groups $\bP O(3) \cong O(3) / \{ \pm 1 \} \cong SO(3)$.
\end{rem}

\section{A review of the stable homotopy theory of surfaces and proof of Theorem \ref{Main}}

In this section, we give a brief survey of the modern homotopy theory of surface bundles developped by Tillmann, Madsen, Weiss and Galatius. Let us first discuss the oriented case. A good survey can be found in \cite{GMT}, which also contains references to all relevant papers.\footnote{The paper \cite{GMT} uses a different notation: they denote $\MTSO (2)$ by $\mathbb{CP}^{\infty}_{-1}$.} A new proof of the main results which can be generalized to the non-oriented case can be found in \cite{GMTW}.

Consider the universal complex line bundle $L \to BSO(2)$. There does not exist a vector bundle $V$ such that $V \oplus L$ is trivial. But we can define an additive inverse $L^{\bot}$ of $L$ as a \emph{stable vector bundle}. The \emph{Madsen--Tillmann spectrum} $\MTSO (2)$ is by definition the Thom spectrum of $L^{\bot}$. For any oriented surface bundle $E \to B$, there exists a natural map 
\[
\alpha_E: B \to \loopinf_{0} \MTSO (2)
\]
into the unit component of the infinite loop space of the Madsen--Tillmann spectrum. In particular, it can be defined for the universal oriented surface bundle with fibers a surface $F_g$ of genus $g$. We obtain a universal map
\[
\alpha_g: B \Diff^+(F_g) \to \loopinf_{0} \MTSO (2).
\]
For all $n > 0$, there exists a cohomology class $y_n \in H^{2n}(\loopinf_{0} \MTSO (2); \bZ)$ such that for any surface bundle
\begin{equation}\label{universalkappa}
\alpha_{E}^{*}(y_n) = \kappa_n (E).
\end{equation}
The rational cohomology of $\Omega_0^\infty \MTSO(2)$ is isomorphic to the polynomial ring $\bQ \, [ y_1, y_2, \ldots]$.
The map $\alpha_g$ induces an isomorphism on homology groups in the stable range, that is,
\[
H_k(\alpha_g):H_k( B \Diff^+(F_g);\bZ) \to H_k(\loopinf_{0} \MTSO (2);\bZ)
\]
is an isomorphism as long as $g \geq 2k+2$.

Similar results are true in the non-oriented case. The Madsen--Tillmann spectrum is replaced by $\MTO (2) $, which is the Thom spectrum of the stable inverse of the universal $2$-dimensional real vector bundle over $BO(2)$. There exists an analogue of the map $\alpha$ for any non-oriented surface bundle. There are classes $x_n \in H^{4n}(\loopinf_{0} \MTO(2); \bZ)$, $ n >0$, such that $\alpha_{E}^{*}(x_n) = \zeta_n (E)$. These things are completely analogous to the oriented case. The rational cohomology ring of $\loopinf \MTO (2)$ is isomorphic to the polynomial ring $\bQ \, [ x_1, x_ 2, \ldots]$.

The analogue of the Madsen--Weiss theorem is also true in the non-oriented case. More precisely
\begin{equation}\label{wahlgmtw}
H_k (\alpha; \bZ): H_k(B \Diff(S_g);\bZ) \to H_k(\loopinf_{0} \MTO (2); \bZ)
\end{equation}
is an isomorphism as long as $4k + 3 \leq g$. The proof of this theorem consists of two parts: one part is the proof of the analogue of the Harer--Ivanov stability theorem in the non-oriented case and was done by Wahl \cite{Wahl}. The other part is the determination of the homotopy type of an appropriate cobordism category by Galatius, Madsen, Tillmann and Weiss \cite{GMTW}.

\begin{proof}[Proof of Theorem \ref{Main}]
This is now straightforward. We assume that the universal class $\zeta_{n}$ is divisible in the stable range. It follows, by (\ref{wahlgmtw}) that the class $ x_n \in H^{4n}(\loopinf MTO(2); \bZ)$ is also divisible. We have seen in Proposition \ref{g0notdiv} that the image of $x_n \in H^{4n}(BSO(3); \bZ)$ under the map $\alpha: BSO(3) \to \loopinf \MTO(2)$ is not divisible. This is a contradiction.
\end{proof}

\end{document}